\documentclass[12pt]{amsart}

\usepackage{framed}
\usepackage{braket}
\usepackage{amsmath,amssymb,amsthm}
\usepackage{color}
\usepackage{bm}
\usepackage[all]{xy}
\usepackage{amscd}
\usepackage{graphicx}
\usepackage{ascmac}
\usepackage{BOONDOX-frak, mathrsfs}
\usepackage[top=30truemm,bottom=30truemm,left=25truemm,right=25truemm]{geometry}
\usepackage{mathtools}
\mathtoolsset{showonlyrefs=true}
\usepackage{tikz}
\usetikzlibrary{cd}

\makeatletter
\@addtoreset{equation}{section}

\makeatother

\makeatletter
\@namedef{subjclassname@2020}{\textup{2020} Mathematics Subject Classification}
\makeatother

\newcommand{\Z}{\mathbb{Z}}

\newcommand{\Q}{\mathbb{Q}}

\newcommand{\cI}{\mathcal{I}}

\newcommand{\cL}{\mathcal{L}}

\newcommand{\wtil}[1]{\widetilde{#1}}
\newcommand{\ol}[1]{\overline{#1}}

\DeclareMathOperator{\pd}{pd}

\DeclareMathOperator{\Cl}{Cl}
\DeclareMathOperator{\ord}{ord}

\let\oldenumerate\enumerate
\renewcommand{\enumerate}{
   \oldenumerate
   \setlength{\itemsep}{1pt}
   \setlength{\parskip}{0pt}
   \setlength{\parsep}{0pt}
}
\let\olditemize\itemize
\renewcommand{\itemize}{
   \olditemize
   \setlength{\itemsep}{1pt}
   \setlength{\parskip}{0pt}
   \setlength{\parsep}{0pt}
}

\theoremstyle{plain}
\newtheorem{thm}{Theorem}[section]
\newtheorem{lem}[thm]{Lemma}

\newtheorem{prop}[thm]{Proposition}

\theoremstyle{definition}
\newtheorem{defn}[thm]{Definition}

\newtheorem{rem}[thm]{Remark}

\newcommand{\bE}{\mathbb{E}}

\newcommand{\sfr}{\mathsf{r}}

\DeclareMathOperator{\Div}{Div}
\DeclareMathOperator{\Pic}{Pic}
\DeclareMathOperator{\Jac}{Jac}

\DeclareMathOperator{\proj}{proj}
\DeclareMathOperator{\Stab}{Stab}

\title
[Kida's formula for graphs]
{Kida's formula for graphs with ramifications}
\author[T. Kataoka]{Takenori Kataoka}
\address{Department of Mathematics, Faculty of Science Division II, Tokyo University of Science.
1-3 Kagurazaka, Shinjuku-ku, Tokyo 162-8601, Japan}
\email{tkataoka@rs.tus.ac.jp}
\date{\today}

\begin{document}

\maketitle

\begin{abstract}
Recently Iwasawa theory for graphs is developing.
A significant achievement includes an analogue of Iwasawa class number formula, which describes the asymptotic growth of the numbers of spanning trees for $\Z_p$-coverings of graphs.
Moreover, an analogue of Kida's formula concerning the behavior of the $\lambda$- and $\mu$-invariants is obtained for unramified coverings.
In this paper, we establish Kida's formula for possibly ramified coverings.
\end{abstract}

\section{Introduction}\label{sec:intro}

\subsection{Classical Iwasawa theory}\label{ss:bac1}

We begin with a brief review of classical Iwasawa theory.
Let $p$ be a prime number.
In classical Iwasawa theory, we mainly deal with a $\Z_p$-extension $K_{\infty}/K$ of number fields, that is,
\[
K = K_0 \subset K_1 \subset K_2 \subset \cdots,
\]
where $K_n/K$ is a $\Z/p^n\Z$-extension.
The so-called Iwasawa class number formula due to Iwasawa \cite[Theorem 11]{Iwa59} (see also Washington \cite[Theorem 13.13]{Was97}) claims
\[
\ord_p(\# \Cl(K_n)) = \lambda n + \mu p^n + \nu,
\quad
n \gg 0
\]
for some integers $\lambda = \lambda(K_{\infty}/K) \geq 0$, $\mu = \mu(K_{\infty}/K) \geq 0$, and $\nu$, which are called the Iwasawa invariants.
Here, $\Cl(K_n)$ denotes the class group of $K_n$ and $\ord_p(-)$ denotes the $p$-adic valuation that is normalized by $\ord_p(p) = 1$.

Kida \cite{Kid80} obtained a formula that describes the behavior of the Iwasawa invariants when $K_{\infty}/K$ varies.
To be concrete, let $G$ be a finite $p$-group and $\wtil{K}_{\infty}/\wtil{K}$ be another $\Z_p$-extension such that $\wtil{K}_n/K$ is a $\Z/p^n \Z \times G$-extension for any $n \geq 0$.
Then conceptually the formula claims that $\mu(K_{\infty}/K) = 0$ is equivalent to $\mu(\wtil{K}_{\infty}/\wtil{K}) = 0$, in which case $\lambda(\wtil{K}_{\infty}/\wtil{K})$ can be described by using $\lambda(K_{\infty}/K)$ together with the ramification information in $\wtil{K}_{\infty}/K_{\infty}$.
(To be precise, the original Kida's formula only concerns the minus components for the cyclotomic $\Z_p$-extensions of CM-fields.)

\subsection{Iwasawa theory for graphs}\label{ss:bac2}

The main object in this paper is graphs, which we assume to be finite and connected in this introduction.
For a graph $X$, let $\kappa(X)$ be the number of spanning trees of $X$.
By Kirchhoff's theorem, the number $\kappa(X)$ is equal to the order of the Jacobian group $\Jac(X)$ of $X$ (see \S \ref{ss:graphs}).
Note that $\Jac(X)$ is also known as the Picard group of degree zero, the sandpile group, and the critical group.
This interpretation enables us to study $\kappa(X)$ by an algebraic investigation of $\Jac(X)$.

Let $X_{\infty}/X$ be an {\it unramified} (or {\it unbranched}) $\Z_p$-covering of graphs, that is, we are given a family of graphs
\[
X = X_0 \leftarrow X_1 \leftarrow X_2 \leftarrow \cdots,
\]
where $X_n/X$ is an unramified $\Z/p^n\Z$-covering.
This $X_n$ is called the $n$-th layer of $X_{\infty}/X$.
Then an analogue of the Iwasawa class number formula, due to Gonet \cite[Theorem 1.1]{Gon22} or McGown--Valli\'{e}res \cite[Theorem 6.1]{MV24} (see also the author's article \cite[\S 8.1]{Kata_21}), claims
\[
\ord_p(\kappa(X_n)) = \lambda n + \mu p^n + \nu,
\quad
n \gg 0
\]
for some integers $\lambda = \lambda(X_{\infty}/X) \geq 0$, $\mu = \mu(X_{\infty}/X) \geq 0$, and $\nu$.

To discuss Kida's formula, let $G$ be a finite $p$-group.
Let $\wtil{X}_{\infty}/\wtil{X}$ be another unramified $\Z_p$-covering such that the $n$-th layer $\wtil{X}_n$ is an {\it unramified} $\Z/p^n \Z \times G$-covering of $X$ for any $n \geq 0$.
Then the analogue of Kida's formula, due to Ray--Valli\'{e}res \cite[Theorem 4.1]{RV22} (see also \cite[\S 8.2]{Kata_21}), says that $\mu(X_{\infty}/X) = 0$ is equivalent to $\mu(\wtil{X}_{\infty}/\wtil{X}) = 0$, in which case we have
\[
\lambda(\wtil{X}_{\infty}/\wtil{X}) + 1 = [\wtil{X}_{\infty}: X_{\infty}] (\lambda(X_{\infty}/X) + 1)
\]
with $[\wtil{X}_{\infty}: X_{\infty}] = \# G$.
Compared to the original formula of Kida, the ramification terms disappear naturally because $\wtil{X}_{\infty}/X_{\infty}$ is assumed to be unramified.

\subsection{Main theorem}\label{ss:intro_th}

As we have emphasized, Iwasawa theory for graphs has been developed mainly for unramified coverings so far.
In a latest paper of Gambheera--Valli\'{e}res \cite{GV}, they initiated an Iwasawa theoretic study of {\it ramified} (or {\it branched}) coverings of graphs (see \S \ref{ss:cov}).
As a fundamental result, for a possibly ramified $\Z_p$-covering $X_{\infty}/X$, they proved that the aforementioned analogue of the Iwasawa class number formula is still valid literally (see Theorem \ref{thm:ICNF}).

In this paper, we establish Kida's formula for graphs with ramifications.
As above, let $G$ be a finite $p$-group and we consider $\Z_p$-coverings $X_{\infty}/X$ and $\wtil{X}_{\infty}/\wtil{X}$ such that $\wtil{X}_n/X$ is a $\Z/p^n \Z \times G$-covering for any $n \geq 0$.
Let us stress that we allow the ramifications both in $\wtil{X}_{\infty}/X_{\infty}$ and in $X_{\infty}/X$.
The main theorem is the following:

\begin{thm}\label{thm:main_A}
We have $\mu(\wtil{X}_{\infty}/\wtil{X}) = 0$ if and only if we have $\mu(X_{\infty}/X) = 0$ and $(\star)$ holds:
\begin{quote}
$(\star)$
any vertex of $X$ is either ramified in $X_{\infty}/X$ or unramified in $\wtil{X}_{\infty}/X_{\infty}$.
\end{quote}
If these equivalent conditions hold, then we have
\[
\lambda(\wtil{X}_{\infty}/\wtil{X}) + 1
= [\wtil{X}_{\infty}: X_{\infty}] (\lambda(X_{\infty}/X) + 1)
 - \sum_{v \in V_X} n_v(\wtil{X}_{\infty}/X) (m_v(\wtil{X}_{\infty}/X_{\infty}) - 1).
\]
Here, $V_X$ denotes the set of vertices of $X$, $m_v(\wtil{X}_{\infty}/X_{\infty})$ denotes the ramification index of $v$ in $\wtil{X}_{\infty}/X_{\infty}$, and $n_v(\wtil{X}_{\infty}/X)$ denotes the number of vertices of $\wtil{X}_{\infty}$ lying over $v$.
\end{thm}

\begin{rem}\label{rem:main_valid}
Theorem \ref{thm:main_A} is formulated symbolically, and to be precise it is better to rephrase it by using only finite graphs $\wtil{X}_n$ and $X_n$ instead of $\wtil{X}_{\infty}$ and $X_{\infty}$.
The degree $[\wtil{X}_{\infty}: X_{\infty}]$ is understood to be $\# G$ (which equals the degree $[\wtil{X}_n: X_n]$ for any $n$).
Let $v$ be a vertex of $X$.
\begin{itemize}
\item
Using the ramification indices $m_v(X_n/X) \geq 1$, we say $v$ is ramified in $X_{\infty}/X$ if $m_v(X_n/X) > 1$ for some $n$.
\item
Using the ramification indices $m_v(\wtil{X}_n/X_n) \geq 1$, which are decreasing with respect to $n$, we define $m_v(\wtil{X}_{\infty}/X_{\infty})$ as its limit.
We say $v$ is unramified in $\wtil{X}_{\infty}/X_{\infty}$ if $m_v(\wtil{X}_{\infty}/X_{\infty}) = 1$.
\item
Using the numbers $n_v(\wtil{X}_n/X)$ of vertices of $\wtil{X}_n$ lying above $v$, which are increasing with respect to $n$, we define $n_v(\wtil{X}_{\infty}/X)$ as its limit.
\end{itemize}
In general, $n_v(\wtil{X}_{\infty}/X)$ can be infinite.
However, in this case, $v$ is unramified in $X_{\infty}/X$, so condition $(\star)$ implies that $m_v(\wtil{X}_{\infty}/X_{\infty}) = 1$ and we set
\[
n_v(\wtil{X}_{\infty}/X) (m_v(\wtil{X}_{\infty}/X_{\infty}) - 1) = 0.
\]
This is why the formula makes sense.
\end{rem}

\begin{rem}
(1)
If $\wtil{X}_{\infty}/X_{\infty}$ is unramified, then condition $(\star)$ holds, no matter how $X_{\infty}/X$ is ramified.
Therefore, we have the equivalence between $\mu = 0$ for $\wtil{X}_{\infty}/\wtil{X}$ and for $X_{\infty}/X$, in which case
\[
\lambda(\wtil{X}_{\infty}/\wtil{X}) + 1
= [\wtil{X}_{\infty}: X_{\infty}] (\lambda(X_{\infty}/X) + 1)
\]
holds.
This formula is of the same form as the previous one for the unramified case.

(2)
In the original Kida's formula concerning class groups, the terms corresponding to $n_v(\wtil{X}_{\infty}/X)$ are automatically finite since we deal with cyclotomic $\Z_p$-extensions of number fields, in which all finite primes are finitely split.
Contrary to this, Iwasawa \cite[Theorem 1]{Iwa73} constructed $\Z_p$-extensions of number fields with $\mu > 0$ by using the anti-cyclotomic $\Z_p$-extensions of imaginary quadratic fields, in which finite primes often splits completely.
Moreover, Iwasawa \cite[Theorem 2]{Iwa73} also showed that $\mu = 0$ is retained as long as a condition like $(\star)$ holds.
These observations imply that condition $(\star)$ is natural.
\end{rem}

\subsection{Organization of this paper}\label{ss:org}

In \S \ref{sec:notation}, we introduce basic notations concerning graphs.
In \S \ref{sec:funct}, we study precise functoriality properties of Picard groups, which is essential to prove the main theorem.
Then in \S \ref{sec:main}, we prove the main theorem.
Finally, in \S \ref{sec:eg}, we observe very simple examples.

\section{Preliminaries}\label{sec:notation}

In \S \ref{ss:graphs}, we fix our convention concerning graphs and then define the Jacobian groups.
In \S \ref{ss:cov}, we define Galois coverings of graphs with ramifications.
In \S\S \ref{ss:vol_graphs}--\ref{ss:Gal_vol_graphs}, we show that Galois coverings arise as the derived graphs of voltage graphs.

\subsection{Graphs and Jacobian groups}\label{ss:graphs}

We first introduce the notions of graphs and the Jacobian groups.
See \cite[\S \S 2.1--2.2]{GV} or the author's article \cite[\S 2]{Kata_21} for the details.
We use Serre's formalism of graphs \cite[Chapter I, \S 2.1]{Ser80}.

\begin{defn}
A finite graph $X$ consists of a finite set $V_X$ of vertices, a finite set $\bE_X$ of edges, an involution on $\bE_X$ without fixed points, denoted by $e \mapsto \ol{e}$, and two maps $s, t: \bE_X \to V_X$ satisfying $s(\ol{e}) = t(e)$ and $t(\ol{e}) = s(e)$.
\end{defn}

Each $e \in \bE_X$ is regarded as an edge that connects $s(e)$ to $t(e)$, and $\ol{e}$ is the opposite of $e$.
For each $v \in V_X$, we write $\bE_{X, v}$ for the set of $e \in \bE_X$ such that $s(e) = v$.
Note that we allow multi-edges and loops, so our graphs may be called multigraphs.
We will usually study connected graphs, which means that any two vertices can be connected by a path of edges.

\begin{defn}
Let $X$ be a finite connected graph.
We define the divisor group of $X$ as the free $\Z$-module on the set $V_X$, i.e., 
\[
\Div(X) = \bigoplus_{v \in V_X} \Z[v].
\]
We define a $\Z$-homomorphism $\cL_X: \Div(X) \to \Div(X)$, called the Laplacian operator, by
\[
\cL_X([v]) = \sum_{e \in \bE_{X, v}} \left([v] - [t(e)] \right),
\quad
v \in V_X.
\]
We define the Picard group $\Pic(X)$ of $X$ as the cokernel of $\cL_X$.

Let $\deg_X: \Div(X) \to \Z$ be the degree map, i.e., $\deg_X$ is the $\Z$-homomorphism that sends any $[v]$ to $1$.
We clearly have $\deg_X \circ \cL_X = 0$, so $\deg_X$ factors through $\Pic(X)$.
Then we define the Jacobian group $\Jac(X)$ as the kernel of the induced map $\deg_X: \Pic(X) \to \Z$.
\end{defn}

It is known that $\Jac(X)$ is a finite abelian group.
Indeed, by Kirchhoff's theorem, the order of $\Jac(X)$ is equal to the number $\kappa(X)$ of spanning trees of $X$.
It also follows that the kernel of $\cL_X$ is a free $\Z$-module of rank one with a basis $\sum_{v \in V_X} [v]$.
Therefore, we have two fundamental exact sequences
\begin{equation}\label{eq:Pic_defn}
0 \to \Z \overset{\iota_X}{\to} \Div(X) \overset{\cL_X}{\to} \Div(X) \to \Pic(X) \to 0,
\end{equation}
where $\iota_X$ sends $1$ to $\sum_{v \in V_X} [v]$, and
\begin{equation}\label{eq:Jac_defn}
0 \to \Jac(X) \to \Pic(X) \overset{\deg_X}{\to} \Z \to 0.
\end{equation}

\subsection{Galois coverings of graphs}\label{ss:cov}

We introduce the notion of (possibly ramified) Galois coverings, basically following \cite[\S 3.1]{GV}.
We begin with the definition of morphisms of graphs.
Let $X$ and $Y$ be finite graphs.

\begin{defn}[{\cite[Definition 2.1]{GV}}]\label{defn:morph}
A morphism $f: Y \to X$ of graphs consists of maps $f_V: V_Y \to V_X$ and $f_{\bE}: \bE_Y \to \bE_X$ such that
\[
f_V(s(e)) = s(f_{\bE}(e)),
\quad
f_V(t(e)) = t(f_{\bE}(e)),
\quad
f_{\bE}(\ol{e}) = \ol{f_{\bE}(e)}
\]
for any $e \in \bE_Y$.
\end{defn}

\begin{defn}[{\cite[Definition 3.1]{GV} or Sunada \cite[page 69]{Sun13}}]\label{defn:cov}
A (ramified) covering $f: Y \to X$ is a morphism of graphs that satisfies the following condition:
\begin{quote}
For each vertex $w \in V_Y$, there is a positive integer $m_w(f)$ such that the map
\[
f_{\bE}: \bE_{Y, w} \to \bE_{X, f_V(w)}
\]
is $m_w(f)$-to-one.
\end{quote}
We often write $Y/X$ for this covering, making $f$ implicit.
\end{defn}

A vertex $w \in V_Y$ is said to be lying above $v \in V_X$ if $f_V(w) = v$.
The number $m_w(Y/X) = m_v(f)$ is referred to as the ramification index of $w$.
Naturally we say $w$ is unramified if $m_w(Y/X) = 1$, and a vertex $v$ of $X$ is unramified if all vertices of $Y$ lying above $v$ are unramified.
An unramified covering is by definition a covering such that any vertices are unramified.

Now we go on to the definition of Galois coverings.
The general definition does not explicitly written in \cite{GV}; instead, they only deal with derived graphs associated to voltage graphs (see \S \ref{ss:vol_graphs}).
In this paper, we introduce a reasonable definition of Galois coverings.
We will see that actually all of them can be constructed as derived graphs (Proposition \ref{prop:der_fin}).

\begin{defn}
Let $\Gamma$ be a finite group.
A $\Gamma$-covering $Y/X$ (or a Galois covering with Galois group $\Gamma$) is a morphism $f: Y \to X$ equipped with an action of $\Gamma$ on $Y$ satisfying the following:
\begin{itemize}
\item
The action of $\Gamma$ on $Y$ respects $f: Y \to X$, so the fibers $f_V^{-1}(v) \subset V_Y$ and $f_\bE^{-1}(e) \subset \bE_Y$ are $\Gamma$-stable for each $v \in V_X$ and $e \in \bE_X$.
\item
For each $v \in V_X$, the action of $\Gamma$ on $f_V^{-1}(v)$ is transitive.
\item
For each $e \in \bE_X$, the action of $\Gamma$ on $f_{\bE}^{-1}(e)$ is transitive and free.
\end{itemize}
\end{defn}

In other words, the action of $\Gamma$ on $\bE_Y$ is free and $f$ induces an isomorphism between the quotient graph $\Gamma \backslash Y$, defined in an obvious way, and $X$.

It is easy to see that a $\Gamma$-covering $Y/X$ is actually a covering (cf.~\cite[Proposition 3.3(1)]{GV}).
The ramification index $m_w(Y/X)$ coincides with the order of the stabilizer subgroup $\Stab_{\Gamma}(w)$ of $w$.
This depends only on the vertex $v = f_V(w) \in V_X$ since the action of $\Gamma$ on $f_V^{-1}(v)$ is transitive.
We then define the ramification index of $v \in V_X$ as $m_v(Y/X) = m_w(Y/X)$ by using any $w \in V_Y$ lying above $v$.

Finally, let us introduce the notion of infinite Galois coverings.

\begin{defn}
Let $\Gamma$ be a profinite group.
A $\Gamma$-covering of a finite graph $X$ is an inverse system $(X_U)_U$ of graphs, indexed by the open normal subgroups $U$ of $\Gamma$ (the order $\preceq$ is defined by $U \preceq V$ if $U \supset V$), equipped with an action of $\Gamma$ that induces a $\Gamma/U$-covering structure on $X_U/X$ for any $U$.
\end{defn}

For instance, for $\Gamma = \Z_p$, the open subgroups are precisely $\{\Gamma^{p^n}\}_{n \geq 0}$, so a $\Gamma$-covering of $X$ consists of a system of graphs
\[
X = X_0 \leftarrow X_1 \leftarrow X_2 \leftarrow \cdots
\]
such that $X_n/X$ is a $\Z/p^n\Z$-covering.

\subsection{Voltage graphs and derived graphs}\label{ss:vol_graphs}

To deal with (finite or infinite) Galois coverings, it is convenient to use voltage graphs and derived graphs.
See \cite[\S 5]{Kata_21} for the unramified cases, and \cite[\S 4]{GV} for the ramified cases.

Firstly, we introduce the voltage graphs.

\begin{defn}
A voltage graph $(X, \Gamma, \alpha)$ consists of a finite graph $X$, a group $\Gamma$, and a map $\alpha: \bE_X \to \Gamma$ satisfying
\[
\alpha(\ol{e}) = \alpha(e)^{-1}
\]
for any $e \in \bE_X$.
To handle the ramifications, we consider an extra structure $\cI = ( I_v \subset \Gamma)_{v \in V_X}$, which is a family of subgroups.
By abuse of notation, we also call $(X, \Gamma, \alpha, \cI)$ a voltage graph.
\end{defn}

\begin{defn}\label{defn:vol_funct}
Let $(X, \Gamma, \alpha, \cI)$ be a voltage graph.
Let $H$ be a normal subgroup of $\Gamma$.
Then the induced voltage graph $(X, \Gamma/H, \alpha_H, \cI_H)$ is defined as follows:
Letting $\pi: \Gamma \to \Gamma/H$ be the natural projection,
we define $\alpha_H = \pi \circ \alpha: \bE_X \to \Gamma/H$ and $\cI_H = (I_{v, H})_{v \in V_X}$ with $I_{v, H} = \pi(I_v) = I_v H/H \subset \Gamma/H$.
\end{defn}

Now we introduce the derived graphs.

\begin{defn}\label{defn:der_gr_fin}
Let $(X, \Gamma, \alpha, \cI)$ be a voltage graph with $\Gamma$ finite.
We construct a finite graph $X(\Gamma, \cI)$ (the map $\alpha$ is implicit), called the derived graph, as follows.
The set of vertices and edges are defined by
\[
V_{X(\Gamma, \cI)} = \coprod_{v \in V_X} (\Gamma/I_v \times \{v\}),
\quad
\bE_{X(\Gamma, \cI)} = \Gamma \times \bE_X.
\]
For an edge $(\gamma, e)$, we set
\[
s((\gamma, e)) = (\gamma I_{s(e)}, s(e)),
\quad
t((\gamma, e)) = (\gamma \alpha(e) I_{t(e)}, t(e)),
\quad
\ol{(\gamma, e)} = (\gamma \alpha(e), \ol{e}).
\]
\end{defn}

The group $\Gamma$ acts on the graph $X(\Gamma, \cI)$ from the left in the natural way.
We also have a natural morphism $X(\Gamma, \cI) \to X$ by projections to the second components.
It is quite easy to see that this morphism is indeed a $\Gamma$-covering.
More generally, given a normal subgroup $H$ of $\Gamma$, setting $(X, \Gamma/H, \alpha_H, \cI_H)$ as in Definition \ref{defn:vol_funct}, we have an $H$-covering $X(\Gamma, \cI) \to X(\Gamma/H, \cI_H)$.

\begin{defn}\label{defn:der_gr_infin}
Let $(X, \Gamma, \alpha, \cI)$ be a voltage graph with $\Gamma$ profinite.
For each open normal subgroup $U$ of $\Gamma$, we have a voltage graph $(X, \Gamma/U, \alpha_U, \cI_U)$ by Definition \ref{defn:vol_funct}, so we can construct the derived graph $X(\Gamma/U, \cI_U)$ by Definition \ref{defn:der_gr_fin}.
We write $X(\Gamma, \cI) = (X(\Gamma/U, \cI_U))_U$ for the $\Gamma$-covering of $X$ obtained in this way.
\end{defn}

\subsection{Galois coverings from voltage graphs}\label{ss:Gal_vol_graphs}

Now let us show that any Galois covering arises from a voltage graph.
We begin with the finite case.

\begin{prop}\label{prop:der_fin}
Let $f: Y \to X$ be a $\Gamma$-covering with $\Gamma$ finite.
Then there is a voltage graph structure $(X, \Gamma, \alpha, \cI)$ such that $Y \simeq X(\Gamma, \cI)$ as $\Gamma$-coverings of $X$.
\end{prop}

\begin{proof}
We construct a voltage graph $(X, \Gamma, \alpha, \cI)$ in the following way.
The construction depends on two choices:
\begin{itemize}
\item
For each $v \in V_X$, choose a vertex $\wtil{v} \in V_Y$ such that $f_V(\wtil{v}) = v$.
\item
For each $e \in \bE_X$, choose an edge $\wtil{e} \in \bE_Y$ such that $f_{\bE}(\wtil{e}) = e$ and $s(\wtil{e}) = \wtil{s(e)}$.
\end{itemize}
Then we define $\alpha$ and $\cI = (I_v)_{v \in V_X}$ as follows:
\begin{itemize}
\item
Let $\alpha(e) \in \Gamma$ be the unique element such that $\alpha(e) \cdot \wtil{\ol{e}} = \ol{\wtil{e}}$.
\item
Let $I_v \subset \Gamma$ be the stabilizer of $\wtil{v}$.
\end{itemize}
Note that $\alpha(e)$ exists uniquely because both $\wtil{\ol{e}}$ and $\ol{\wtil{e}}$ are in the fiber of $\ol{e}$ with respect to $f$ and the action of $\Gamma$ on such a fiber is free and transitive.
Then we have $t(\wtil{e}) = \alpha(e) \wtil{t(e)}$ by
\[
t(\wtil{e}) = s(\ol{\wtil{e}}) = s(\alpha(e) \cdot \wtil{\ol{e}}) = \alpha(e) s(\wtil{\ol{e}}) =\alpha(e) \wtil{s(\ol{e})} = \alpha(e) \wtil{t(e)}.
\]

We construct a morphism $\phi = (\phi_V, \phi_{\bE}): X(\Gamma, \cI) \to Y$ by
\[
\phi_V((\gamma I_v, v)) = \gamma \wtil{v},
\quad
\phi_{\bE}((\gamma, e)) = \gamma \wtil{e}.
\]
The well-definedness of $\phi_V$ follows from the definition of $I_v$.
It is straightforward to see that $\phi$ is indeed a morphism:
\[
\phi_V(s(\gamma, e)) 
= \phi_V((\gamma I_{s(e)}, s(e)))
= \gamma \wtil{s(e)} 
= \gamma s(\wtil{e}) 
= s(\gamma \wtil{e})
= s(\phi_{\bE}((\gamma, e))),
\]
\[
\phi_V(t(\gamma, e)) 
= \phi_V((\gamma \alpha(e) I_{t(e)}, t(e)))
= \gamma \alpha(e) \wtil{t(e)} 
= \gamma t(\wtil{e}) 
= t(\gamma \wtil{e})
= t(\phi_{\bE}((\gamma, e))),
\]
and
\[
\ol{\phi_{\bE}((\gamma, e))} 
= \ol{\gamma \wtil{e}}
= \gamma \ol{\wtil{e}}
= \gamma \alpha(e) \wtil{\ol{e}}
= \phi_{\bE}((\gamma \alpha(e), \ol{e}))
= \phi_{\bE}(\ol{(\gamma, e)}).
\]
Moreover, by the definition of the $\Gamma$-coverings, both $\phi_V$ and $\phi_{\bE}$ are bijective, so $\phi$ is an isomorphism.
\end{proof}

\begin{prop}\label{prop:der_infin}
Let $\Gamma$ be a profinite group and $(X_U)_U$ be a $\Gamma$-covering of $X$.
Then there is a voltage graph structure $(X, \Gamma, \alpha, \cI)$ such that $(X_U)_U$ is isomorphic to $X(\Gamma, \cI)$ as $\Gamma$-coverings of $X$.
\end{prop}

\begin{proof}
As in the proof of Proposition \ref{prop:der_fin}, we firstly make two choices:
\begin{itemize}
\item
For each vertex $v$ of $X$ and $U \subset \Gamma$, choose a vertex $\wtil{v}^U$ of $X_U$ that goes to $v$.
Moreover, we impose the compatibility with respect to $U$, that is, $\wtil{v}^V$ goes to $\wtil{v}^U$ whenever $U \supset V$.
\item
For each edge $e$ of $X$, choose an edge $\wtil{e}^U$ of $X_U$ that goes to $e$ and satisfies $s(\wtil{e}^U) = \wtil{s(e)}^U$.
Moreover, we impose the compatibility with respect to $U$, that is, $\wtil{e}^V$ goes to $\wtil{e}^U$ whenever $U \supset V$.
\end{itemize}
Then, for each $U$, we define $\alpha_U: \bE_X \to \Gamma/U$ and $I_{v, U} \subset \Gamma/U$ as in the proof of Proposition \ref{prop:der_fin} applied to $X_U/X$.
It is straightforward to see that $\alpha_U$ and $I_{v, U}$ are compatible with respect to $U$, so we may define $\alpha: \bE_X \to \Gamma$ and $I_v \subset \Gamma$ as the limits.
By the proof of Proposition \ref{prop:der_fin}, for any $U$, we have $X_U \simeq X(\Gamma/U, \cI_U)$ as $\Gamma/U$-coverings.
Therefore, $(X_U)_U \simeq X(\Gamma, \cI)$ as $\Gamma$-coverings.
\end{proof}

\section{Functorialities of Picard groups}\label{sec:funct}

The behavior of the Picard groups with respect to coverings plays important roles in the proof of Theorem \ref{thm:main_A}.
We first consider general coverings in \S \ref{ss:cov_funct}.
Then we will consider finite and infinite Galois coverings in \S \ref{ss:fin_funct} and \S \ref{ss:profin}, respectively.

\subsection{Coverings}\label{ss:cov_funct}

Let $f: Y \to X$ be a covering (Definition \ref{defn:cov}) between finite connected graphs.
We define the degree of the covering $Y/X$ as the number
\[
[Y: X] := \sum_{w \in f_V^{-1}(v)} m_w(f),
\]
which is independent of $v \in V_X$ (see \cite[\S 3.1]{GV}; here we need the connectedness of $X$).
Note that, for any edge $e \in \bE_X$, we have
\[
[Y: X] = \# (f_{\bE}^{-1}(e)),
\]
that is, the fiber of an arbitrary edge of $X$ consists of $[Y: X]$ edges of $Y$.

As in \cite[\S 3.2]{GV}, let
\[
f_*, f_{\sfr}: \Div(Y) \to \Div(X)
\]
be the $\Z$-homomorphisms such that
\[
f_*([w]) = [f_V(w)],
\quad
f_{\sfr}([w]) = m_w(f) [f_V(w)]
\]
for each $w \in V_Y$.

\begin{prop}\label{prop:funct_cov}
The exact sequences \eqref{eq:Pic_defn} for $X$ and $Y$ satisfy a  commutative diagram
\begin{equation}\label{eq:Pic_func}
\xymatrix{
	0 \ar[r]
	& \Z \ar[r]^-{\iota_X} \ar[d]_{[Y: X]}
	& \Div(Y) \ar[r]^{\cL_{Y}} \ar[d]_{f_{\sfr}}
	& \Div(Y) \ar@{->>}[d]^{f_*} \ar[r]
	& \Pic(Y) \ar[r] \ar@{->>}[d]
	& 0 \\
	0 \ar[r]
	& \Z \ar[r]_-{\iota_Y}
	& \Div(X) \ar[r]_{\cL_X} 
	& \Div(X) \ar[r] 
	& \Pic(X) \ar[r]
	& 0,
}
\end{equation}
where the rightmost vertical map between the Picard groups is induced by $f_*$.
\end{prop}

\begin{proof}
This proposition follows from direct computations.
The commutativity of the left one follows from
\[
f_{\sfr} \left( \sum_{w \in V_Y} [w] \right)
= \sum_{w \in V_Y} m_w(f)[f(w)]
= \sum_{v \in V_X}  \left( \sum_{w \in f_V^{-1}(v)} m_w(f) \right) [v]
= [Y: X] \sum_{v \in V_X} [v].
\]
The commutativity of the middle square is \cite[Proposition 3.4]{GV}.
Indeed, for each $w \in V_Y$, setting $v = f_V(w) \in V_X$, we see
\[
f_* \circ \cL_Y([w])
= f_* \left(\sum_{e \in \bE_{Y, w}} ([w] - [t(e)]) \right)
= \sum_{e \in \bE_{Y, w}} ([f_V(w)] - [f_V(t(e))])
= \sum_{e \in \bE_{Y, w}} ([v] - [t(f_{\bE}(e))])
\]
and
\[
\cL_X \circ f_{\sfr}([w])
 = m_w(f) \cL_X([v])
= m_w(f) \sum_{e \in \bE_{X, v}} ([v] - [t(e)]).
\]
Since $f_{\bE}: \bE_{Y, w} \to \bE_{X, v}$ is $m_w(f)$-to-one, these coincide.
\end{proof}

\subsection{Finite Galois coverings}\label{ss:fin_funct}

Let $(X, \Gamma, \alpha, \cI)$ be a voltage graph with $\Gamma$ finite.

Set $\Z[\Gamma/I_v] = \Z[\Gamma] \otimes_{\Z[I_v]} \Z$, which we regard as a left $\Z[\Gamma]$-module.
Then we may identify
\begin{equation}\label{eq:Div_d}
\Div(X(\Gamma, \cI)) 
\simeq \bigoplus_{v \in V_X} \Z[\Gamma/I_v] [v]
\simeq \bigoplus_{v \in V_X} \Z[\Gamma/I_v]
\end{equation}
as $\Z[\Gamma]$-modules.
Here, the first isomorphism is given by $(\gamma I_v, v) \leftrightarrow \gamma [v]$ and the second by simply omitting $[v]$ to ease the notation.
Then, associated to the endomorphism $\cL_{X(\Gamma, \cI)}$ on $\Div(X(\Gamma, \cI))$, we introduce an endomorphism $\cL_{X, \Gamma, \cI}$ on $\bigoplus_{v \in V_X} \Z[\Gamma/I_v] [v]$ (or $\bigoplus_{v \in V_X} \Z[\Gamma/I_v]$).
In other words, $\cL_{X, \Gamma, \cI}$ is defined by a commutative diagram
\[
\xymatrix{
	\Div(X(\Gamma, \cI)) \ar[r]^{\cL_{X(\Gamma, \cI)}} \ar[d]_{\simeq}
	& \Div(X(\Gamma, \cI)) \ar[d]^{\simeq}\\
	\bigoplus_{v \in V_X} \Z[\Gamma/I_v] [v] \ar[r]_{\cL_{X, \Gamma, \cI}}
	& \bigoplus_{v \in V_X} \Z[\Gamma/I_v] [v].
}
\]

\begin{rem}
It is straightforward to see that $\cL_{X, \Gamma, \cI}$ may be explicitly defined by 
\[
\cL_{X, \Gamma, \cI}([v]) 
= N_{I_v} \cdot \sum_{e \in \bE_{X, v}} \left([v] - \alpha(e) [t(e)] \right),
\quad
v \in V_X.
\]
Here, $N_{I_v} = \sum_{\sigma \in I_v} \sigma \in \Z[I_v]$ denotes the norm element of $I_v$.
However, this explicit formula will not be used.
\end{rem}

\begin{prop}\label{prop:funct_fin_Gal}
Let $(X, \Gamma, \alpha, \cI)$ be a voltage graph with $\Gamma$ finite.
Suppose that $X(\Gamma, \cI)$ is connected (so $X$ is also connected).

\begin{itemize}
\item[(1)]
We have an exact sequence
\[
0 \to \Z \to \bigoplus_{v \in V_X} \Z[\Gamma/I_v] 
\overset{\cL_{X, \Gamma, \cI}}{\to} \bigoplus_{v \in V_X} \Z[\Gamma/I_v] 
\to \Pic(X(\Gamma, \cI)) \to 0.
\]
\item[(2)]
Let $H$ be a normal subgroup of $\Gamma$.
We introduce a voltage graph $(X, \Gamma/H, \alpha_H, \cI_H)$ as in Definition \ref{defn:vol_funct}.
Then we have a commutative diagram
\begin{equation}\label{eq:Pic_func_3}
\xymatrix{
	0 \ar[r]
	& \Z \ar[r] \ar[d]_{\# H}
	& \bigoplus_{v \in V_X} \Z[\Gamma/I_v] \ar[r]^{\cL_{X, \Gamma, \cI}} \ar[d]_{\beta}
	& \bigoplus_{v \in V_X} \Z[\Gamma/I_v] \ar@{->>}[d]^{\proj} \ar[r]
	& \Pic(X(\Gamma, \cI)) \ar[r] \ar@{->>}[d]
	& 0 \\
	0 \ar[r]
	& \Z \ar[r]
	& \bigoplus_{v \in V_X} \Z[\Gamma/I_v H] \ar[r]_{\cL_{X, \Gamma/H, \cI_H}} 
	& \bigoplus_{v \in V_X} \Z[\Gamma/I_v H] \ar[r] 
	& \Pic(X(\Gamma/H, \cI_H)) \ar[r]
	& 0.
}
\end{equation}
Here, the map $\proj$ is the natural projection, and the map $\beta$ sends $1 \in \Z[\Gamma/I_v]$ to $\#(H \cap I_v) \in \Z[\Gamma/I_v H]$.
\end{itemize}
\end{prop}

\begin{proof}
(1)
This is simply a reformulation of \eqref{eq:Pic_defn}.

(2)
Let $f: X(\Gamma, \cI) \to X(\Gamma/H, \cI_H)$ be the $H$-covering.
It is quite easy to see that the following diagrams commute:
\[
\xymatrix{
	\bigoplus_{v \in V_X} \Z[\Gamma/I_v] \ar[r]^{\simeq} \ar[d]_{\beta}
	& \Div(X(\Gamma, \cI)) \ar[d]^{f_{\sfr}}\\
	\bigoplus_{v \in V_X} \Z[\Gamma/I_v H] \ar[r]_{\simeq}
	& \Div(X(\Gamma/H, \cI_H)),
}
\quad
\xymatrix{
	\bigoplus_{v \in V_X} \Z[\Gamma/I_v] \ar[r]^{\simeq} \ar@{->>}[d]_{\proj}
	& \Div(X(\Gamma, \cI)) \ar[d]^{f_*}\\
	\bigoplus_{v \in V_X} \Z[\Gamma/I_v H] \ar[r]_{\simeq}
	& \Div(X(\Gamma/H, \cI_H)),
}
\]
where the horizontal isomorphisms are \eqref{eq:Div_d}.
Then the claimed digram is simply a reformulation of Proposition \ref{prop:funct_cov} applied to $f$.
\end{proof}

This description of $\beta$ in this proposition will be crucial in what follows.

\subsection{Infinite Galois coverings}\label{ss:profin}

As in \cite[\S 5.2]{Kata_21}, we shall define the Picard group of an infinite Galois covering by using the projective limit.
Let $\Lambda$ be a compact flat $\Z$-algebra; in the applications we will set $\Lambda = \Z_p$.
This $\Lambda$ will be used to make the projective limit functor exact.

\begin{defn}\label{defn:Pic_inf}
Let $(X, \Gamma, \alpha, \cI)$ be a voltage graph with $\Gamma$ profinite.
Suppose that $X(\Gamma/U, \cI_U)$ are connected for all $U$.
By Proposition \ref{prop:funct_fin_Gal}(2), the Picard groups $\Pic(X(\Gamma/U, \cI_U))$ form a projective system.
We define
\[
\Pic_{\Lambda}(X(\Gamma, \cI)) = \varprojlim_U \Lambda \otimes_{\Z} \Pic(X(\Gamma/U, \cI_U)).
\]
This is a module over the Iwasawa algebra $\Lambda[[\Gamma]] = \varprojlim_U \Lambda[\Gamma/U]$.
\end{defn}

For any profinite group $\Gamma$, we define an ideal ${}_{\Gamma} \Lambda$ of $\Lambda$ by
\[
{}_{\Gamma} \Lambda = \bigcap_U [\Gamma: U] \Lambda,
\]
where $U$ runs over the open normal subgroups of $\Gamma$.
For instance, if $\Lambda = \Z_p$, we have ${}_{\Gamma} \Lambda = 0$ if and only if the order of $\Gamma$ is divisible by $p^{\infty}$.

\begin{prop}\label{prop:funct_infin_Gal}
Let $(X, \Gamma, \alpha, \cI)$ be a voltage graph with $\Gamma$ profinite.
Suppose that $X(\Gamma/U, \cI_U)$ are connected for all $U$.
\begin{itemize}
\item[(1)]
We have an exact sequence
\begin{equation}\label{eq:Pic_func_5}
0 
\to {}_{\Gamma} \Lambda
\to \bigoplus_{v \in V_X} {}_{I_v} \Lambda[[\Gamma/I_v]] \overset{\cL_{X, \Gamma, \cI}}
\to \bigoplus_{v \in V_X} \Lambda[[\Gamma/I_v]]
\to \Pic_{\Lambda}(X(\Gamma, \cI)) \to 0,
\end{equation}
where $\cL_{X, \Gamma, \cI}$ is introduced in the proof below.
\item[(2)]
Let $H$ be a closed normal subgroup of $\Gamma$.
Let $(X, \Gamma/H, \alpha_H, \cI_H)$ be as in Definition \ref{defn:vol_funct}.
Then we have a commutative diagram
\begin{equation}\label{eq:Pic_func_2}
\small
\xymatrix{
	0 \ar[r]
	& {}_{\Gamma} \Lambda \ar[r] \ar[d]
	& \bigoplus_{v \in V_X} {}_{I_v} \Lambda[[\Gamma/I_v]] \ar[r]^{\cL_{X, \Gamma, \cI}} \ar[d]
	& \bigoplus_{v \in V_X} \Lambda[[\Gamma/I_v]] \ar@{->>}[d] \ar[r]
	& \Pic_{\Lambda}(X(\Gamma, \cI)) \ar[r] \ar@{->>}[d]
	& 0 \\
	0 \ar[r]
	& {}_{\Gamma/H} \Lambda \ar[r]
	& \bigoplus_{v \in V_X} {}_{I_v H} \Lambda[[\Gamma/I_v H]] \ar[r]_{\cL_{X, \Gamma/H, \cI_H}} 
	& \bigoplus_{v \in V_X} \Lambda[[\Gamma/I_v H]] \ar[r] 
	& \Pic_{\Lambda}(X(\Gamma/H, \cI_H)) \ar[r]
	& 0. \\
}
\end{equation}
Here, all the vertical arrows are the natural ones:
The left two ones are induced by the inclusions ${}_{\Gamma} \Lambda \hookrightarrow {}_{\Gamma/H} \Lambda$ and ${}_{I_v} \Lambda \hookrightarrow {}_{I_vH} \Lambda$.
\end{itemize}
\end{prop}

\begin{proof}
By Proposition \ref{prop:funct_fin_Gal}(2), for open normal subgroups $V, U$ of $\Gamma$ with $U \supset V$, we have
\begin{equation}\label{eq:Pic_func_4}
\xymatrix{
	0 \ar[r]
	& \Z \ar[r] \ar[d]_{[U: V]}
	& \bigoplus_{v \in V_X} \Z[\Gamma/I_v V] \ar[r]^{\cL_{X, \Gamma/V, \cI_V}} \ar[d]_{\beta}
	& \bigoplus_{v \in V_X} \Z[\Gamma/I_v V] \ar@{->>}[d]^{\proj} \ar[r]
	& \Pic(X(\Gamma/V, \cI_V)) \ar[r] \ar@{->>}[d]
	& 0 \\
	0 \ar[r]
	& \Z \ar[r]
	& \bigoplus_{v \in V_X} \Z[\Gamma/I_v U] \ar[r]_{\cL_{X, \Gamma/U, \cI_U}} 
	& \bigoplus_{v \in V_X} \Z[\Gamma/I_v U] \ar[r] 
	& \Pic(X(\Gamma/U, \cI_U)) \ar[r]
	& 0.
}
\end{equation}
Here, $\beta$ on the $v$-component sends $1$ to
\[
\# ((U/V) \cap (I_v V/V)) = [U \cap I_v V: V].
\]
By taking the tensor product with $\Lambda$ over $\Z$ and then taking the projective limit with respect to $U$, we obtain claim (1).
For claim (2), we only have to take the limit of the above diagram (after $\Lambda \otimes_{\Z}(-)$) with the restriction $U \supset H$.
\end{proof}

\section{Kida's formula}\label{sec:main}

In \S \ref{ss:ICNF}, we review the Iwasawa class number formula in our settings.
Then in \S \ref{ss:Kida_pf}, we prove our main theorem (Theorem \ref{thm:main_A}).

\subsection{Iwasawa class number formula}\label{ss:ICNF}

From now on, we fix a prime number $p$ and work over $\Lambda = \Z_p$.

Let $\Gamma$ be a profinite group that is isomorphic to $\Z_p$.
For a finitely generated torsion $\Z_p[[\Gamma]]$-module $M$, we have the $\lambda$, $\mu$-invariants $\lambda(M)$, $\mu(M)$, which are non-negative integers.
They are usually defined by using the structure theorem for $\Z_p[[\Gamma]]$-modules (see Neukirch--Schmidt--Wingberg \cite[Definition (5.3.9)]{NSW08}).
Instead of the usual definition, we only recall the following properties:
\begin{itemize}
\item
We have $\mu(M) = 0$ if and only if $M$ is finitely generated over $\Z_p$.
\item
We have $\lambda(M) = \dim_{\Q_p} (\Q_p \otimes_{\Z_p} M)$.
\end{itemize}

Now the Iwasawa class number formula is the following:

\begin{thm}[{Gambheera--Valli\'{e}res \cite[Theorem A]{GV}}]\label{thm:ICNF}
Let $X_{\infty}/X$ be a $\Z_p$-covering of finite connected graphs and $X_n$ be its $n$-th layer.
Then we have
\[
\ord_p(\kappa(X_n)) = \lambda n + \mu p^n + \nu,
\quad
n \gg 0
\]
with
\[
\lambda = \lambda(X_{\infty}/X) = \lambda(\Pic_{\Z_p}(X_{\infty})) - 1,
\quad
\mu = \mu(X_{\infty}/X) = \mu(\Pic_{\Z_p}(X_{\infty}))
\]
and some integer $\nu = \nu(X_{\infty}/X)$.
\end{thm}

\begin{proof}
The result \cite[Theorem A]{GV} claims the formula with $\lambda = \lambda(\Jac_{\Z_p}(X_{\infty}))$ and $\mu = \mu(\Jac_{\Z_p}(X_{\infty}))$, where $\Jac_{\Z_p}(X_{\infty})$ is defined as the projective limit of $\Z_p \otimes_{\Z} \Jac(X_n)$ (cf.~Definition \ref{defn:Pic_inf}).
We have an exact sequence
\[
0 \to \Jac_{\Z_p}(X_{\infty}) \to \Pic_{\Z_p}(X_{\infty}) \to \Z_p \to 0
\]
induced by \eqref{eq:Jac_defn}.
This implies
\[
\lambda(\Jac_{\Z_p}(X_{\infty})) = \lambda(\Pic_{\Z_p}(X_{\infty})) - 1,
\quad
\mu(\Jac_{\Z_p}(X_{\infty})) = \mu(\Pic_{\Z_p}(X_{\infty})),
\]
so Theorem \ref{thm:ICNF} is also valid.
\end{proof}

\subsection{Kida's formula}\label{ss:Kida_pf}

A key to prove the main theorem is the following algebraic proposition.

\begin{prop}[{\cite[Proposition 8.5]{Kata_21}}]\label{prop:alg_Kida}
Let $\Gamma$ be a profinite group that is isomorphic to $\Z_p$, and $G$ a finite $p$-group.
Let $M$ be a finitely generated torsion $\Z_p[[\Gamma]][G]$-module.
We write $M_G$ for the $G$-coinvariant of $M$.
\begin{itemize}
\item[(1)]
We have $\mu(M) = 0$ if and only if $\mu(M_G) = 0$.
\item[(2)]
Suppose that we have $\pd_{\Z_p[[\Gamma]][G]}(M) \leq 1$, where $\pd$ denotes the projective dimension.
If the equivalent conditions in (1) hold, then we have $\lambda(M) = (\# G) \lambda(M_G)$.
\end{itemize}
\end{prop}

\begin{proof}
Because of its importance, we include a rough sketch of the proof.
Claim (1) follows from Nakayama's lemma over the local ring $\Z_p[G]$.
Under the assumptions in claim (2), the $\Z_p[G]$-module $M$ is indeed free of finite rank.
Then the equality $\lambda(M) = (\# G) \lambda(M_G)$ is clear.
\end{proof}

The author \cite{Kata_27} extended this proposition to perfect complexes.
The result is so useful that we can deduce analogues of Kida's formula in various arithmetic situations.
However, we do not need such an extension in this paper.

We are in a position to prove Theorem \ref{thm:main_A}.
By Proposition \ref{prop:der_infin}, the situation in Theorem \ref{thm:main_A} can be realized by a voltage graph.
To be precise, let $\wtil{\Gamma}$ be a profinite group of the form
\[
\wtil{\Gamma} = \Gamma \times G
\]
with $\Gamma$ isomorphic to $\Z_p$ and $G$ a finite $p$-group.
Let $(X, \wtil{\Gamma}, \wtil{\alpha}, \wtil{\cI})$ be a voltage graph with $\wtil{\cI} = ( \wtil{I_v} \subset \wtil{\Gamma})_{v \in V_X}$.
Let us assume that $X(\wtil{\Gamma}/U, \wtil{\cI}_U)$ is connected for any open normal subgroup $U$ of $\wtil{\Gamma}$.
We also obtain $(X, \Gamma, \alpha, \cI)$ (write $\cI = (I_v \subset \Gamma)_{v \in V_X}$) by using Definition \ref{defn:vol_funct} applied to the closed subgroup $G$ of $\wtil{\Gamma}$.
We set
\[
X_{\infty} = X(\Gamma, \cI),
\quad
\wtil{X}_{\infty} = X(\wtil{\Gamma}, \wtil{\cI}).
\]
Now Theorem \ref{thm:main_A} can be rephrased as follows:

\begin{thm}\label{thm:main}
We have $\mu(\Pic_{\Z_p}(\wtil{X}_{\infty})) = 0$ if and only if we have $\mu(\Pic_{\Z_p}(X_{\infty})) = 0$ and $(\star)$ holds:
\begin{quote}
$(\star)$
for any $v \in V_X$, the group $\wtil{I_v}$ is either infinite or trivial.
\end{quote}
If these equivalent conditions hold, then we have
\[
\lambda(\Pic_{\Z_p}(\wtil{X}_{\infty}))
= (\# G) \lambda(\Pic_{\Z_p}(X_{\infty}))
- \sum_{v \in V_X} [\wtil{\Gamma}: \wtil{I_v}] (\# (G \cap \wtil{I_v}) - 1).
\]
\end{thm}

As in Remark \ref{rem:main_valid}, if $[\wtil{\Gamma}: \wtil{I_v}]$ is infinity, condition $(\star)$ implies $\wtil{I_v}$ is trivial, so we set $[\wtil{\Gamma}: \wtil{I_v}] (\# (G \cap \wtil{I_v}) - 1) = 0$ in this formula.

\begin{proof}
Set $V_X^0 = \{v \in V_X \mid I_v = 0 \}$.
Note that 
\[
I_v = 0 
\quad \Leftrightarrow \quad \text{$I_v$ is finite}
\quad \Leftrightarrow \quad \text{$\wtil{I_v}$ is finite},
\]
where the first equivalence holds since $I_v$ is $p$-torsion-free.
The motivation for this definition is that we have ${}_{I_v} \Lambda = {}_{\wtil{I_v}} \Lambda = 0$ if (and only if) $v \not \in V_X^0$.
Then Proposition \ref{prop:funct_infin_Gal}(2) yields a commutative diagram
\begin{equation}\label{eq:Pic_func_6}
\xymatrix{
	0 \ar[r]
	& \bigoplus_{v \in V_X^0} \Z_p[[\wtil{\Gamma}/\wtil{I_v}]] \ar[r]^-{\cL_{X, \wtil{\Gamma}, \wtil{\cI}}} \ar[d]_{\beta}
	& \bigoplus_{v \in V_X} \Z_p[[\wtil{\Gamma}/\wtil{I_v}]] \ar@{->>}[d]^{\proj} \ar[r]
	& \Pic_{\Z_p}(\wtil{X}_{\infty}) \ar[r] \ar@{->>}[d]
	& 0 \\
	0 \ar[r]
	& \bigoplus_{v \in V_X^0} \Z_p[[\Gamma]] \ar[r]_-{\cL_{X, \Gamma, \cI}} 
	& \bigoplus_{v \in V_X} \Z_p[[\Gamma/I_v]] \ar[r] 
	& \Pic_{\Z_p}(X_{\infty}) \ar[r]
	& 0.
}
\end{equation}
Here, $\beta$ on the $v$-component sends $1$ to $\# \wtil{I_v}$.
By taking the $G$-coinvariant of the upper sequence, we obtain
\begin{equation}\label{eq:Pic_func_7}
\xymatrix{
	0 \ar[r]
	& \bigoplus_{v \in V_X^0} \Z_p[[\Gamma]] \ar[r]^-{\ol{\cL_{X, \wtil{\Gamma}, \wtil{\cI}}}} \ar@{^(->}[d]_{\ol{\beta}}
	& \bigoplus_{v \in V_X} \Z_p[[\Gamma/I_v]] \ar@{=}[d] \ar[r]
	& \Pic_{\Z_p}(\wtil{X}_{\infty})_G \ar[r] \ar@{->>}[d]
	& 0 \\
	0 \ar[r]
	& \bigoplus_{v \in V_X^0} \Z_p[[\Gamma]] \ar[r]_-{\cL_{X, \Gamma, \cI}} 
	& \bigoplus_{v \in V_X} \Z_p[[\Gamma/I_v]] \ar[r] 
	& \Pic_{\Z_p}(X_{\infty}) \ar[r]
	& 0.
}
\end{equation}
Here, $\ol{\beta}$ is the injective homomorphism that is the multiplication by $\# \wtil{I_v}$ on the $v$-component.
This also explains why the induced homomorphism $\ol{\cL_{X, \wtil{\Gamma}, \wtil{\cI}}}$ is injective.
By the snake lemma, we obtain an exact sequence
\[
0 \to \bigoplus_{v \in V_X^0} (\Z_p/\# \wtil{I_v})[[\Gamma]]
\to \Pic_{\Z_p}(\wtil{X}_{\infty})_G
\to \Pic_{\Z_p}(X_{\infty})
\to 0.
\]
Note that $\mu((\Z_p/\# \wtil{I_v})[[\Gamma]]) = 0$ if and only if $\# \wtil{I_v} = 1$.
Thus, we have $\mu(\Pic_{\Z_p}(\wtil{X}_{\infty})_G) = 0$ if and only if $\mu(\Pic_{\Z_p}(X_{\infty})) = 0$ and $\# \wtil{I_v} = 1$ for any $v \in V_X^0$.
Thanks to Proposition \ref{prop:alg_Kida}(1), this shows the first claim on the equivalence concerning $\mu = 0$.

In what follows, we assume the equivalent conditions.
Then the $\ol{\beta}$ in \eqref{eq:Pic_func_7} is the identity map, so we have $\Pic_{\Z_p}(\wtil{X}_{\infty})_G \simeq \Pic_{\Z_p}(X_{\infty})$.
However, Proposition \ref{prop:alg_Kida}(2) is not applicable to $\Pic_{\Z_p}(\wtil{X}_{\infty})$.
This is because the condition $\pd \leq 1$ does not hold in general.
To remedy this, we use a slightly modified version of the Picard groups that satisfies $\pd \leq 1$.

We define
\[
\cL_{X, \wtil{\Gamma}, \wtil{\cI}}':
\bigoplus_{v \in V_X^0} \Z_p[[\wtil{\Gamma}]] 	
\to \bigoplus_{v \in V_X^0} \Z_p[[\wtil{\Gamma}]]
\]
as the composition of $\cL_{X, \wtil{\Gamma}, \wtil{\cI}}$ with the projection to the $V_X^0$-components.
Define $\Pic_{\Z_p}'(\wtil{X}_{\infty})$ as the cokernel of this $\cL_{X, \wtil{\Gamma}, \wtil{\cI}}'$.
Clearly we have an exact sequence
\[
0 \to \bigoplus_{v \in V_X \setminus V_X^0} \Z_p[\wtil{\Gamma}/\wtil{I_v}]
\to \Pic_{\Z_p}(\wtil{X}_{\infty})
\to \Pic_{\Z_p}'(\wtil{X}_{\infty})
\to 0.
\]
Therefore, we have
\[
\lambda(\Pic_{\Z_p}(\wtil{X}_{\infty}))
= \lambda(\Pic_{\Z_p}'(\wtil{X}_{\infty})) 
+ \sum_{v \in V_X \setminus V_X^0} [\wtil{\Gamma}: \wtil{I_v}]
\]
and $\mu(\Pic_{\Z_p}'(\wtil{X}_{\infty})) = \mu(\Pic_{\Z_p}(\wtil{X}_{\infty})) = 0$.
We also define $\Pic_{\Z_p}'(X_{\infty})$ in the same way: 
We introduce $\cL_{X, \Gamma, \cI}'$ that is the projection of $\cL_{X, \Gamma, \cI}$ to the $V_X^0$-components and then define $\Pic_{\Z_p}'(X_{\infty})$ as its cokernel.
Then the same reasoning shows
\[
\lambda(\Pic_{\Z_p}(X_{\infty}))
= \lambda(\Pic_{\Z_p}'(X_{\infty}))
+ \sum_{v \in V_X \setminus V_X^0} [\Gamma: I_v].
\]

Since $\Pic_{\Z_p}'(\wtil{X}_{\infty})$ is torsion, the homomorphism $\cL_{X, \wtil{\Gamma}, \wtil{\cI}}'$ is injective (cf.~\cite[Lemma A.3]{Kata_21}), so we have
\[
\pd_{\Z_p[[\wtil{\Gamma}]]}(\Pic_{\Z_p}'(\wtil{X}_{\infty})) \leq 1.
\]
By the same reasoning as the original Picard groups, since the homomorphism $\ol{\beta}$ is the identity, we have
\[
\Pic_{\Z_p}'(\wtil{X}_{\infty})_G \simeq \Pic_{\Z_p}'(X_{\infty}).
\]
Now we are able to apply Proposition \ref{prop:alg_Kida}(2) to $\Pic_{\Z_p}'(\wtil{X}_{\infty})$.
As a result, we obtain
\[
\lambda(\Pic_{\Z_p}'(\wtil{X}_{\infty})) = (\# G) \lambda(\Pic_{\Z_p}'(X_{\infty})).
\]
By combining these formulas, we obtain
\[
\lambda(\Pic_{\Z_p}(\wtil{X}_{\infty}))
- \sum_{v \in V_X \setminus V_X^0} [\wtil{\Gamma}: \wtil{I_v}]
= (\# G) \left( \lambda(\Pic_{\Z_p}(X_{\infty})) 
- \sum_{v \in V_X \setminus V_X^0} [\Gamma: I_v] \right).
\]
An elementary observation shows
\[
(\# G) [\Gamma: I_v] - [\wtil{\Gamma}: \wtil{I_v}]
= [\wtil{\Gamma}: \wtil{I_v}] (\# (G \cap I_v) - 1).
\]
This completes the proof of Theorem \ref{thm:main}.
\end{proof}

\section{Examples}\label{sec:eg}

Let $m$ be a positive integer.
Let $X$ be the cycle graph with $m$ vertices:
The vertices are
\[
V_X = \{v_1, \dots, v_m\}
\]
and the edges are
\[
\bE_X = \{e_1, \dots, e_m, \ol{e_1}, \dots, \ol{e_m}\}
\]
such that $s(e_i) = v_i$ and $t(e_i) = v_{i + 1}$ for $1 \leq i \leq m$ ($v_{m+1}$ is understood to be $v_1$).

Let $\Gamma \simeq \Z_p$ and $G$ a cyclic group of order $p$.
Let
\[
\alpha: \bE_X \to \wtil{\Gamma} = \Gamma \times G
\]
be a voltage assignment such that the components of $\alpha(e_1)$ are generators of $\Gamma$ and $G$, and $\alpha(e_2), \dots, \alpha(e_m)$ are all the unit element.

We consider three choices of $\wtil{\cI}$:
\begin{itemize}
\item[(a)]
$\wtil{I_v} = G$ for any $v$.
\item[(b)]
$\wtil{I_v} = \Gamma$ for any $v$.
\item[(c)]
$\wtil{I_v} = \Gamma \times G$ for any $v$.
\end{itemize}
Note that condition $(\star)$ holds for (b) and (c) but does not hold for (a).
In each case, we set 
\[
X_{\infty} = X(\Gamma, \cI),
\quad
\wtil{X}_{\infty} = X(\wtil{\Gamma}, \wtil{\cI}),
\quad
X_n = X(\Gamma/\Gamma^{p^n}, \cI_{\Gamma^{p^n}}),
\quad
\wtil{X}_n = X(\wtil{\Gamma}/\Gamma^{p^n}, \wtil{\cI}_{\Gamma^{p^n}}).
\]
As will be clear, $\wtil{X}_n$ are indeed connected.

To compute the Iwasawa $\lambda$, $\mu$-invariants, it is convenient to prepare a lemma.
For positive integers $m$ and $N$, we define a graph $Y_{m, N}$ as follows.
We set $Y_{m, 1} = X$, which is the cycle graph with $m$ vertices.
Then we set $Y_{m, N}$ to be the ``totally ramified'' covering of $Y_{m, 1}$ of degree $N$.
More precisely, $Y_{m, N}$ has $m$ vertices, say $v_1, \dots, v_m$, and $m N$ unoriented edges, and for each $1 \leq i \leq m$, exactly $N$ edges connects $v_i$ and $v_{i+1}$.

\begin{lem}\label{lem:kappa}
We have $\kappa(Y_{m, N}) = m N^{m-1}$.
\end{lem}

\begin{proof}
This is easily proved by counting the number of spanning trees.
\end{proof}

\begin{table}
\centering
\caption{}
\renewcommand{\arraystretch}{1.3} 
\begin{tabular}{c||c|c|c}\label{table:1}
& Case (a) & Case (b) & Case (c) \\
\hline \hline
$X_n$ & $\simeq Y_{mp^n, 1}$ & $\simeq Y_{m, p^n}$ & $\simeq Y_{m, p^n}$ \\
$\kappa(X_n)$ & $mp^n$ & $m p^{n(m-1)}$ & $m p^{n(m-1)}$ \\
$\mu(X_{\infty}/X)$ & $0$ & $0$ & $0$ \\
$\lambda(X_{\infty}/X)$ & $1$& $m - 1$ & $m - 1$ \\
\hline
$\wtil{X}_n$ & $\simeq Y_{mp^n, p}$ & $\simeq Y_{m p, p^n}$ & $\simeq Y_{m, p^{n+1}}$ \\
$\kappa(\wtil{X}_n)$ & $mp^n \cdot p^{mp^n-1}$ & $m p \cdot p^{n(m p -1)}$ & $m \cdot p^{(m-1) (n+1)}$ \\
$\mu(\wtil{X}_{\infty}/\wtil{X})$ & $m$ & $0$ & $0$ \\
$\lambda(\wtil{X}_{\infty}/\wtil{X})$ & $1$ & $pm - 1$ & $m - 1$ \\
\end{tabular}
\end{table}

Using this lemma, we can directly compute the Iwasawa $\lambda$, $\mu$-invariants of $X_{\infty}/X$ and of $\wtil{X}_{\infty}/\wtil{X}$ for cases (a), (b), and (c).
The results are summarized in Table \ref{table:1} .
For example, in case (a), we have $X_n \simeq Y_{m p^n, 1}$, so $\kappa(X_n) = m p^n$ by Lemma \ref{lem:kappa}.
This implies $\mu = 0$, $\lambda = 1$ for $X_{\infty}/X$.
The other entries are filled in a similar way (note that $X_{\infty}/X$ for (b) is the same as that for (c)).

Let us discuss the validity of Theorem \ref{thm:main_A}.
In case (a), $\mu = 0$ is not retained because condition ($\star$) does not hold, as the theorem predicts.
In case (b), we have
\[
(pm - 1) +1  = p \cdot ((m - 1) + 1)
\]
as predicted.
In case (c), we have
\[
(m- 1) + 1 = p \cdot ((m - 1) + 1) - \sum_{i=1}^m 1 \cdot (p-1)
\]
as predicted.

%

\section*{Acknowledgments}

I am grateful to Daniel Valli\`eres and Rusiru Gambheera for giving helpful comments on an earlier version of this paper.
This work is supported by JSPS KAKENHI Grant Number 22K13898.

{
\bibliographystyle{abbrv}
\bibliography{biblio}

\begin{thebibliography}{10}

\bibitem{GV}
R.~Gambheera and D.~Valli^^c3^^a8res.
\newblock Iwasawa theory for branched $\mathbb{Z}_p$-towers of finite graphs.
\newblock {\em preprint, arXiv:2404.05131}, 2024.

\bibitem{Gon22}
S.~R. Gonet.
\newblock Iwasawa {T}heory of {J}acobians of {G}raphs.
\newblock {\em Algebr. Comb.}, 5(5):827--848, 2022.

\bibitem{Iwa59}
K.~Iwasawa.
\newblock On {$\Gamma $}-extensions of algebraic number fields.
\newblock {\em Bull. Amer. Math. Soc.}, 65:183--226, 1959.

\bibitem{Iwa73}
K.~Iwasawa.
\newblock On the {$\mu $}-invariants of {$Z\sb{\ell}$}-extensions.
\newblock In {\em Number theory, algebraic geometry and commutative algebra, in
  honor of {Y}asuo {A}kizuki}, pages 1--11. Kinokuniya Book Store, Tokyo, 1973.

\bibitem{Kata_21}
T.~Kataoka.
\newblock Fitting ideals of {J}acobian groups of graphs.
\newblock {\em Algebr. Comb.}, 7(3):597--625, 2024.

\bibitem{Kata_27}
T.~Kataoka.
\newblock Kida's formula via {S}elmer complexes.
\newblock {\em preprint, arXiv:2401.07036}, 2024.

\bibitem{Kid80}
Y.~Kida.
\newblock {$l$}-extensions of {CM}-fields and cyclotomic invariants.
\newblock {\em J. Number Theory}, 12(4):519--528, 1980.

\bibitem{MV24}
K.~McGown and D.~Valli\`eres.
\newblock On abelian {$\ell$}-towers of multigraphs {III}.
\newblock {\em Ann. Math. Qu\'{e}.}, 48(1):1--19, 2024.

\bibitem{NSW08}
J.~Neukirch, A.~Schmidt, and K.~Wingberg.
\newblock {\em Cohomology of number fields}, volume 323 of {\em Grundlehren der
  Mathematischen Wissenschaften [Fundamental Principles of Mathematical
  Sciences]}.
\newblock Springer-Verlag, Berlin, second edition, 2008.

\bibitem{RV22}
A.~Ray and D.~Valli\`{e}res.
\newblock An analogue of {K}ida's formula in graph theory.
\newblock {\em preprint, arXiv:2209.04890}, 2022.

\bibitem{Ser80}
J.-P. Serre.
\newblock {\em Trees}.
\newblock Springer-Verlag, Berlin-New York, 1980.
\newblock Translated from the French by John Stillwell.

\bibitem{Sun13}
T.~Sunada.
\newblock {\em Topological crystallography}, volume~6 of {\em Surveys and
  Tutorials in the Applied Mathematical Sciences}.
\newblock Springer, Tokyo, 2013.
\newblock With a view towards discrete geometric analysis.

\bibitem{Was97}
L.~C. Washington.
\newblock {\em Introduction to cyclotomic fields}, volume~83 of {\em Graduate
  Texts in Mathematics}.
\newblock Springer-Verlag, New York, second edition, 1997.

\end{thebibliography}
}

\end{document}